\documentclass{article}

\usepackage{tikz-cd}
\usepackage{comment}
\usepackage[all]{xy}
\usepackage[usenames,dvipsnames]{pstricks}
\usepackage{epsfig}
\usepackage{pst-grad} 
\usepackage{pst-plot} 
\usepackage{ebezier}
\usepackage{latexsym}\usepackage{amsfonts}\usepackage{amssymb}\usepackage{amsthm}
\usepackage{amsmath}\usepackage{newlfont}\usepackage{graphicx}\usepackage{doc}
\usepackage{mathrsfs}
\usepackage[
        colorlinks, citecolor=red,
        backref,
        pdfauthor={MDJ,JK},
        pdftitle={S1}
]{hyperref}

\newtheorem{thrm}{Theorem}

\newtheorem{cor}[thrm]{Corollary}

\newtheorem{rem}[thrm]{Remark}
\newtheorem{prop}[thrm]{Proposition}
\newtheorem{exam}[thrm]{Example}
\newtheorem*{theorem-non}{Theorem}

\def \Dj{\mbox{\raise0.3ex\hbox{-}\kern-0.4em D}}

\newcommand{\NW}{\mathrm{NW}}
\newcommand{\CR}{\mathrm{CR}}
\newcommand{\Fix}{\mathrm{Fix}}

\begin{document}

\title{A note on the non-wandering dynamics in the $n$-fold symmetric products}

\author{Ma\v{s}a \Dj ori\'c\\
Matemati\v{c}ki institut SANU\\
Knez Mihailova 36\\
11000 Beograd\\
Serbia\\
masha@mi.sanu.ac.rs\\
\href{https://orcid.org/0000-0003-1364-2339}{https://orcid.org/0000-0003-1364-2339}
}

\maketitle

\begin{abstract} 
We construct a homeomorphism on a compact metric space whose non-wandering set consists of a single fixed point, while the induced map $F_n(f)$ possesses infinitely many (non-fixed) non-wandering points. 
\end{abstract}

\medskip


\section{Introduction}

Every continuous map $f:X\to X$ on a compact metric space $X$ induces a continuous map (called the induced map) on the hyperspace $2^X$ of all nonempty closed subsets. Its subspace, namely the hyperspace $F_n(X)$ of all nonempty subsets with at most $n$ points (for all $n\in\mathbb{N}$), is of main interest in our paper. The dynamics of induced maps on hyperspaces, and, in particular, on the $n$-fold symmetric product $F_n(X)$, has been studied extensively from the viewpoint of the relation between the dynamics of a map and that of its induced map. Gómez-Rueda, Illanes and Méndez investigated how several properties are transferred between $f$ and $F_n(f)$ in \cite{RIM}. In particular, they showed that the naturally assumed connection fails for the non-wandering set. Their example demonstrates that a finite set of wandering points may become non-wandering in the induced hyperspace. In this note, we show that this phenomenon persists even when the dynamics of the base system is extremely simple. We construct a homeomorphism $f$ on a compact metric space whose non-wandering set consists of a single fixed point, while the induced map $F_n(f)$ possesses a non-fixed non-wandering point. Thus, even the condition $\NW(f)=\Fix(f)$, with $|\NW(f)|=1$ does not imply that $\NW(F_n(f))=\Fix(F_n(f))$. This example also highlights the essential difference between assumptions involving the non-wandering set and stronger hypotheses involving the chain recurrent set in the study of finite symmetric products. More precisely, the author showed in \cite{Dj}, that the following theorem is true:

\begin{thrm}\label{CR}
Let $X$ be a compact metric space and $f:X\to X$ a continuous map such that $\mathcal{CR}(f)=\mathrm{Fix}(f)=\mathrm{NW}(f)$ is a finite set. Then 
$$\mathrm{NW}(F_n(f))=\mathrm{Fix}(F_n(f))=\big\{\{x_1,\ldots,x_k\}\mid 1\leqslant k\leqslant n,\,x_j\in\mathrm{Fix}(f)\big\}.$$
\end{thrm}

The analysis of the non-wandering set of the induced map $F_n(f)$ and other maps is important because it provides insight into the long-term behaviour of a dynamical system. The structure of the non-wandering set of an induced map cannot be inferred directly from that of the underlying map, in general. The example we constructed in this paper shows that the assumption in Theorem \ref{CR} cannot be weakened to the corresponding assumption on the non-wandering set. Beside the interest in the non-wandering set on its own, there are further topics in dynamical systems that require the analysis of the non-wandering set. In particular, in \cite{KP}, authors Kati\'c and Peri\'c generalise a method for computing the polynomial entropy developed by Hauseux and Le Roux in~\cite{HL}, for maps which are continuous and whose non-wandering set is finite. The author used this method for computing the polynomial entropy of the induced map on the $n$-fold symmetric product in \cite{Dj}.

\section{Preliminaries}
Suppose that $X$ is a compact metric space, and $f:X\rightarrow X$ is continuous. For $x\in X$, $\omega$-{\it limit set} is defined as:
$$\omega_f(x):=\left\{y\in X\mid \exists \;\mbox {a strictly increasing sequence}\;\mathbb{N}\ni n_k\to\infty\;\,f^{n_k}(x)\to y\right\},$$
and, for a reversible map $f$, $\alpha$-{\it limit set} is defined as:
$$\alpha_f(x):=\left\{y\in X\mid \exists \;\mbox {a strictly increasing sequence}\;\mathbb{N}\ni n_k\to\infty\;\,f^{-n_k}(x)\to y\right\}.$$
If $X$ is compact, the set $\omega_f(x)$ is nonempty, closed and $f$-invariant. The same is true for $\alpha_f(x)$, when $f$ is reversible.

A point $x\in X$ is \textit{non-wandering} if for every neighborhood $U$ of $x$ there is $n\geqslant1$ such that $f^n(U)\cap U$ is nonempty. We denote the set of all non-wandering points by $\NW(f)$. The set $\NW(f)$ is closed and $f$-invariant. Also, we denote the set of all fixed points by $\mathrm{Fix}(f)$.

In order to define another important invariant subset, let us begin by defining an $\varepsilon$-chain. We say that a tuple $(y_1,\ldots,y_k)$ is an $\varepsilon$-chain if $d(f(y_i),y_{i+1})<\varepsilon$, $1\leqslant i\leqslant k-1$. A point $x\in X$ is $\varepsilon$-pseudo-periodic if there exists an $\varepsilon$-chain that starts and ends at $x$. Finally, a point $x\in X$ is \textit{chain recurrent} if $x$ is $\varepsilon$-pseudo-periodic for all $\varepsilon>0$. We denote the set of all chain recurrent points by $\CR(f)$. The set $\CR(f)$ is a closed and $f$-invariant set which contains $\NW(f)$. 

Let $A$ be a finite set, called {\it alphabet}, whose elements are called {\it symbols}. Denote by $A^\mathbb{Z}$ the set of all two-sided sequences $x=\{x_n\}_{n=-\infty}^\infty$ equipped with the Tychonoff product topology. The space $A^\mathbb{Z}$ is metrizable by the metric:
$$d(x,y):=2^{-m},$$ where
$$m=\min\{|n|\mid x_n\neq y_n\}.$$ With this metric, $A^\mathbb{Z}$ is a compact, totally disconnected metric space without isolated points. The {\it shift map} is defined as:
$$\sigma:A^\mathbb{Z}\to A^\mathbb{Z},\quad\{\sigma(x)\}_n=x_{n+1}$$ and it is a homeomorphism. Any closed subset $\Sigma\subseteq A^\mathbb{Z}$ with $\sigma(\Sigma)=\Sigma$ is said to be a {\it two-sided subshift} or a {\it subshift}, for short.

\subsection{Hyperspaces and induced maps}

For a compact metric space $(X,d)$, the hyperspace $2^X$ is the set of all nonempty closed subsets of $X$. The topology on $2^X$ (namely, the Vietoris topology) is induced by the Hausdorff metric
$$d_H(A,B):=\inf\{\varepsilon>0\mid A\subset U_\varepsilon(B),\;B\subset U_\varepsilon(A)\},$$
where
\begin{equation}\label{eq:neighb}
U_\varepsilon(A):=\{x\in X\mid d(x,A)<\varepsilon\}.
\end{equation}
We consider a closed subspace $F_n(X)$ of $2^X$, $n\geqslant1$, consisting of all finite nonempty subsets of cardinality at most $n$, with the induced metric. The set $F_n(X)$ is called the \textit{$n$-fold symmetric product of $X$}. Both spaces $2^X$ and $F_n(X)$ are compact metric spaces with respect to the Hausdorff metric.

If $f:X\to X$ is continuous, then it induces continuous maps
 $2^f:2^X\to 2^X$ and $F_n(f):F_n(X)\to F_n(X)$, $F_n(f)=2^f\mid_{F_n(X)}$ by:
 $$2^f(A):=\{f(x)\mid x\in A\}.$$
 If $f$ is a homeomorphism, then this is also true for both $2^f$ and $F_n(f)$. Because of the following commutative diagram, it is clear that $(F_n(f),F_n(X))$ is a factor of $(f^{\times n},X^{\times n})$.
 
 \[ \begin{tikzcd}
X^{\times n} \arrow{r}{f^{\times n}} \arrow[swap]{d}{\pi} & X^{\times n} \arrow{d}{\pi} \\%
F_n(X)\arrow{r}{F_n(f)}& F_n(X)
\end{tikzcd}
\]
Here, $X^{\times n}=\underbrace{X\times\dots\times X}_{n}$, $f^{\times n}(x_1,\ldots,x_n)=(f(x_1),\ldots,f(x_n))$ and\\ $\pi(x_1,\ldots,x_n)=\{x_1,\ldots,x_n\},$ which is a continuous surjective map. The metric $\mathbf{d}$ on the product space $X^{\times n}$ is given with:
$$\mathbf{d}((x_1,\ldots,x_n),(y_1,\ldots,y_n))=\max_{1\leqslant i\leqslant n}\{{d(x_i,y_i)}\}.$$ 
Note that all the $p$-metrics, $1\leqslant p\leqslant\infty$ are equivalent on $X^{\times n}$.
\begin{exam}
If $X=[0,1]$, then $F_2(X)$ can be identified with the triangle whose vertices are $(0,0),(0,1),(1,1)$ and $F_1(X)$ is homeomorphic to the line segment joining $(0,0)$ and $(1,1)$. If $X=\mathbb{S}^1$, then $F_2(X)$ is homeomorphic to the M$\ddot{o}$bius band and $F_1(X)$ is homeomorphic to the manifold boundary of the M$\ddot{o}$bius band.
\end{exam}

\section{Main results}

Firstly, let us prove a simple proposition which inspired the construction of our main example.

\begin{prop}\label{switch}
Let $f:X\to X$ be a continuous map on a compact metric space. Suppose $x\neq y$ and there exist $x_n,y_n\in X$ and $m_n\to\infty$ such that $x_n\to x$, $y_n\to y$ and $f^{m_n}(x_n)\to y$, $f^{m_n}(y_n)\to x$. Then $\{x,y\}\in\NW(F_2(f))$.
\end{prop}

\begin{proof}
Let $A=\{x,y\}$ and $A_n=\{x_n,y_n\}$. Since $x_n\to x$ and $y_n\to y$, we have
\[
d_H(A_n,A)
\leqslant
\max\{d(x_n,x),d(y_n,y)\}\rightarrow 0.
\]
Hence $A_n\rightarrow A$ in $F_2(X)$.

Furthermore, $F_2(f)^{m_n}(A_n)=\{f^{m_n}(x_n),f^{m_n}(y_n)\}.$ Since
\[
f^{m_n}(x_n)\to y
\qquad\text{and}\qquad
f^{m_n}(y_n)\to x,
\]
we similarly obtain
\[
d_H\bigl(F_2(f)^{m_n}(A_n),A\bigr)
\leqslant
\max\bigl\{
d(f^{m_n}(x_n),y),
d(f^{m_n}(y_n),x)
\bigr\}
\rightarrow 0.
\]
We conclude that $F_2(f)^{m_n}(A_n)\rightarrow A.$

Let $U$ be an arbitrary neighborhood of $A$ in $F_2(X)$.
For all sufficiently large $n$, both
\[
A_n\in U
\qquad\text{and}\qquad
F_2(f)^{m_n}(A_n)\in U.
\]
Since $A_n\in U$, it follows that $F_2(f)^{m_n}(A_n)\in F_2(f)^{m_n}(U)\cap U$.
Consequently,
\[
F_2(f)^{m_n}(U)\cap U\neq\emptyset.
\]
Since $U$ was arbitrary, $A$ is a non-wandering point of $F_2(f)$. Therefore $\{x,y\}\in\NW(F_2(f))$.
\end{proof}

Let $\Sigma=\{0,1\}^{\mathbb Z}$ be the full two-sided shift endowed with the product topology, and let $\sigma:\Sigma\to\Sigma$ be the shift map. We will now define a closed $\sigma$-invariant set $X\subset \{0,1\}^{\mathbb{Z}}$. Elements of $X$ are two-sided sequences $(x_n)_{n\in\mathbb{Z}}$ having at most one occurrence of $01$ and at most one occurrence of $10$. The elements of $X$ can be described explicitly. They are of the following five types.

\begin{enumerate}
    \item The two constant sequences $\mathbf{0}=0^{\mathbb Z}$ and $\mathbf{1}=1^{\mathbb Z}$.

    \item Sequences with exactly one block of type $01$. For
    $r\in\mathbb Z$, define
    \[
    x^{(r)}_i=
    \begin{cases}
        0, & i<r,\\
        1, & i\geq r.
    \end{cases}
    \]
    Thus
    \[
    x^{(r)}
    =
    \ldots 000\,\big|\,111\ldots,
    \]
    where the first $1$ occurs at coordinate $r$.

    \item Sequences with exactly one block of type $10$. For
    $r\in\mathbb Z$, define
    \[
    y^{(r)}_i=
    \begin{cases}
        1, & i<r,\\
        0, & i\geq r.
    \end{cases}
    \]
    Thus
    \[
    y^{(r)}
    =
    \ldots 111\,\big|\,000\ldots,
    \]
    where the first $0$ occurs at coordinate $r$.

    \item Sequences containing a finite block of $1$'s surrounded by
    $0$'s. For $r\in\mathbb Z$ and $m\in\mathbb N$, define
    \[
    z_{r,m}(i)=
    \begin{cases}
        1, & r\leq i<r+m,\\
        0, & \text{otherwise}.
    \end{cases}
    \]
    Such sequences have exactly one $01$ block and exactly one
    $10$ block.

    \item Sequences containing a finite block of $0$'s surrounded by
    $1$'s. For $r\in\mathbb Z$ and $m\in\mathbb N$, define
    \[
    w_{r,m}(i)=
    \begin{cases}
        0, & r\leq i<r+m,\\
        1, & \text{otherwise}.
    \end{cases}
    \]
    Again, such sequences have exactly one $10$ block and exactly
    one $01$ block.
\end{enumerate}

These are all the elements of $X$, so
\[
X=
\{\mathbf{0},\mathbf{1}\}
\cup
\{x^{(r)}:r\in\mathbb Z\}
\cup
\{y^{(r)}:r\in\mathbb Z\}
\cup
\{z_{r,m}:r\in\mathbb Z,\ m\in\mathbb N\}
\cup
\{w_{r,m}:r\in\mathbb Z,\ m\in\mathbb N\}.
\]

\begin{thrm}
$(X,\sigma)$ is a closed invariant set in $\{0,1\}^{\mathbb{Z}}$ such that $\NW(\sigma|_X)=\Fix(\sigma|_X)$, but $\NW(F_N(\sigma|_X))\neq \Fix(F_N(\sigma|_X))$.
\end{thrm}

\begin{proof}
By construction it is clear that $\sigma(X)=X$. Let us prove that $X$ is closed, or equivalently that $X^{c}$ is open. For an arbitrary $x\in X^{c}$ we have that there are at least two occurrences of $01$ or $10$. Let $x_ix_{i+1}=01$ and $x_jx_{j+1}=01$, for some $i,j\in\mathbb{Z}$. Then the cylinder neighbourhood $C=\{u\in X\mid u_iu_{i+1}=01,u_ju_{j+1}=01\}\subset X^{c}$, so $X^{c}$ is indeed open. 

Let $f=\sigma|_X$. Then $f$ is a homeomorphism on the compact space $X$. It is clear that $\Fix(f)=\{\mathbf{0},\mathbf{1}\}$. We will now show that $\NW(f)=\{\mathbf{0},\mathbf{1}\}$. Let $x\neq \mathbf{0},\mathbf{1}$. Then $x$ contains $01$ or $10$ (or both), so the cylinder $U=\{u\in X\mid u_iu_{i+1}=01\}$ in $X$ is a wandering neighbourhood. Indeed, if $f^k(U)\cap U\neq\emptyset$, then the element in the intersection would have two occurrences of $01$, which is impossible.

For every $n\geq 1$, define $z_n,w_n\in X$ by
\[
(z_n)_i=
\begin{cases}
1, & 0\leq i<n,\\
0, & \text{otherwise},
\end{cases}
\qquad
(w_n)_i=
\begin{cases}
0, & 0\leq i<n,\\
1, & \text{otherwise}.
\end{cases}
\]

Define $x=\ldots 000\mid 111\ldots$ and $y=\ldots 111\mid 000\ldots$, or, more precisely:
\[
x_i=
\begin{cases}
0, & i<0,\\
1, & i\geq 0,
\end{cases}
\qquad
y_i=
\begin{cases}
1, & i<0,\\
0, & i\geq 0.
\end{cases}
\]
Then $x$ and $y$ are wandering points for $f$. Let us show that $A=\{x,y\}\in F_2(X)$ is a non-wandering point of $F_2(f)$, which is obviously not a fixed point of $F_2(f)$. Define $z_n=\ldots 000\mid
\underbrace{11\ldots 1}_{n\text{ times}}
\mid 000\ldots,
$
and
$
w_n=\ldots 111\mid
\underbrace{00\ldots 0}_{n\text{ times}}
\mid 111\ldots.
$ Then $z_n\to x$ and $w_n\to y$, $n\to\infty$ and $f^n(z_n)\to y$, $f^n(w_n)\to x$, $n\to\infty$. We can now apply Proposition \ref{switch} and conclude that $A\in\NW(F_2(f))$, but as we mentioned $A\notin\Fix(F_2(f))$. Note that $A\in \NW(F_N(f))$, for all $N\geqslant2$, which can be shown analogously as for $N=2$.
\end{proof}

\begin{cor}\label{posledica}
There exist a compact metric space $Y$ and a homeomorphism $g:Y\to Y$ such that 
\[ \NW(g)=\Fix(g)=\{p\}, \]
while \[ \NW(F_2(g))\neq\Fix(F_2(g)). \] 
\end{cor}

\begin{proof} Let $(X,f)$ be the dynamical system constructed and used in the previous theorem, and let 
\[ \NW(f)=\Fix(f)=\{\mathbf{0},\mathbf{1}\}. \] 
Define an equivalence relation $\sim$ on $X$ by identifying $\mathbf{0}$ and $\mathbf{1}$ and leaving every other point in a singleton equivalence class. Let 
\[ q:X\to Y=X/{\sim} \] 
be the quotient map and denote \[ p=q(\mathbf{0})=q(\mathbf{1}). \] Since $X$ is a compact metric space and the only nontrivial equivalence class is the finite set $\{\mathbf{0},\mathbf{1}\}$, the quotient space $Y$ is compact and metrizable. Since $\mathbf{0}$ and $\mathbf{1}$ are fixed by $f$, the equivalence relation is $f$-invariant. Hence, $f$ induces a homeomorphism 
\[ g:Y\to Y, \qquad g(q(x))=q(f(x)). \] 
We first show that $\NW(g)=\{p\}.$ Clearly $p$ is fixed, and hence $p\in\NW(g)$. Let $y\in Y\setminus\{p\}$ and let $x$ be its (unique) preimage under $q$. By the construction of $(X,f)$, the point $x$ is wandering. Choose a wandering open neighborhood $U$ of $x$ such that $U\cap\{\mathbf{0},\mathbf{1}\}=\emptyset$. Every projection is an open map, so $V=q(U)$ is an open neighborhood of $y$. Suppose that \[ g^m(V)\cap V\neq\emptyset \] for some $m\neq0$. Then there exist $u,v\in U$ such that $ q(f^m(u))=q(v).$ Since $\mathbf{0}$ and $\mathbf{1}$ are fixed and $u\notin\{\mathbf{0},\mathbf{1}\}$, we have $f^m(u)\notin\{\mathbf{0},\mathbf{1}\}$. The map $q$ is injective outside $\{\mathbf{0},\mathbf{1}\}$, and therefore $ f^m(u)=v.$ Thus \[ f^m(U)\cap U\neq\emptyset, \] contrary to the fact that $U$ is wandering. Hence $V$ is wandering. It follows that $ \NW(g)=\{p\}. $ We conclude that\[ \NW(g)=\Fix(g)=\{p\}. \] 
Let $x,y\in X$ and the sequences $(x_k)$, $(y_k)$ and $(m_k)$ be those from the previous example, so that \[ x_k\to x,\qquad y_k\to y, \] and \[ f^{m_k}(x_k)\to y,\qquad f^{m_k}(y_k)\to x. \] Put \[ \bar x=q(x),\qquad \bar y=q(y). \] Since $x$ and $y$ are distinct from $\mathbf{0}$ and $\mathbf{1}$, and from each other, we have $\bar x\neq\bar y$. By continuity of $q$, \[ q(x_k)\to\bar x,\qquad q(y_k)\to\bar y, \] while \[ g^{m_k}(q(x_k)) =q(f^{m_k}(x_k))\to\bar y \] and \[ g^{m_k}(q(y_k)) =q(f^{m_k}(y_k))\to\bar x. \] Therefore, by the Proposition \ref{switch}, \[ \{\bar x,\bar y\}\in\NW(F_2(g)). \] On the other hand, \[ F_2(g)(\{\bar x,\bar y\}) = \{q(f(x)),q(f(y))\} \neq \{\bar x,\bar y\}, \] since the points involved are distinct from $\mathbf{0},\mathbf{1}$ and the corresponding pair in the original example is not fixed. We conclude \[ \{\bar x,\bar y\}\in \NW(F_2(g))\setminus\operatorname{Fix}(F_2(g)). \] 
\end{proof}

\begin{rem}\label{remarka}
It is worth mentioning that the set $\{q(\sigma^r(x)),q(\sigma^s(y))\}$ is non-wandering in $F_2(Y)$ for all $r,s\in\mathbb{Z}$. Therefore, there are infinitely many non-wandering points in $F_2(Y)$, even though the map $g$ has only one non-wandering (fixed) point. Additionally, there is only one more family of non-wandering points in $F_2(Y)$. Using similar techniques as before, we can prove the following proposition.

\begin{prop}
Let $g:Y\to Y$ be the homeomorphism obtained in the previous
construction by identifying the two fixed points
$\mathbf{0}$ and $\mathbf{1}$ of $(X,f)$, and let $q:X\to Y$ be the quotient map. Denote the unique fixed point of $g$ by $p=q(\mathbf{0})=q(\mathbf{1})$.

Let
\[
x_i=
\begin{cases}
0,& i<0,\\
1,& i\geqslant 0,
\end{cases}
\qquad
y_i=
\begin{cases}
1,& i<0,\\
0,& i\geqslant 0,
\end{cases}
\]
and for $r,s\in\mathbb Z$, put
\[
x_r=q(\sigma^r x),
\qquad
y_s=q(\sigma^s y).
\]
Then
\[
\NW(F_2(g))
=
\bigl\{\{p,u\}|u\in Y\bigr\}
\cup
\bigl\{\{x_r,y_s\}|r,s\in\mathbb Z\bigr\}.
\]
\end{prop}
\end{rem}

\begin{rem}
We mentioned the chain recurrent set of the base map, and the connection to the non-wandering set in the $n-$fold symmetric product hyperspace in Theorem \ref{CR}. Interestingly, the chain recurrent set of the constructed map $g$ in Corollary \ref{posledica} is the entire space $Y$. Namely, take an arbitrary $y\in Y$ and $\varepsilon>0$. Since $g^n(y)\to p$, $n\to\infty$ and $g^{-n}(y)\to p$, $n\to\infty$, we can choose $N$ sufficiently large so that $d(g^{N+1}(y),p)<\varepsilon/2$ and $d(g^{-N}(y),p)<\varepsilon/2$. This ensures that $y,g(y),g^2(y),\ldots,g^{N}(y),g^{-N}(y),g^{-N+1}(y),\ldots,g^{-1}(y),y$ is an $\varepsilon$-chain starting and ending at $y$, proving that every point in $Y$ is chain recurrent.
\end{rem}

\textbf{Acknowledgements}:\\

The author is grateful to Jelena Kati\'c for many constructive discussions regarding non-wandering and chain recurrent sets on hyperspaces.

\end{document}